\documentclass[siggraph, review=false]{acmart}
\usepackage{graphicx}
\usepackage{amsmath}
\usepackage{algorithm}
\usepackage{flushend}
\usepackage{booktabs} % For formal tables
\setcopyright{rightsretained}
\acmDOI{}

\acmISBN{}

\acmConference[PCSC2017]{17th Philippines Computing Science Congress}{March 2017}{University of San Carlos, Cebu, Philippines} 
\acmYear{2017}
\copyrightyear{2017}

\acmPrice{}

\begin{document}
\title{Premagic and Ideal Flow Matrices}

\author{Kardi Teknomo}
\orcid{0000-0003-1123-5924}
\affiliation{%
  \institution{Ateneo de Manila University}
  \streetaddress{}
  \city{Quezon City} 
  \country{Philippines} 
}
\email{teknomo@gmail.com}

\renewcommand{\shortauthors}{Teknomo et. al.}

\begin{abstract}
Several interesting properties of a special type of matrix that has a row sum equal to the column sum are shown with the proofs. Premagic matrix can be applied to strongly connected directed network graph due to its nodes conservation flow. Relationships between Markov Chain, ideal flow and random walk on directed graph are also discussed. 
\end{abstract}

\begin{CCSXML}
	<ccs2012>
	<concept>
	<concept_id>10010147.10010257.10010293.10010309.10010310</concept_id>
	<concept_desc>Computing methodologies~Non-negative matrix factorization</concept_desc>
	<concept_significance>500</concept_significance>
	</concept>
	<concept>
	<concept_id>10003752.10010061.10010065</concept_id>
	<concept_desc>Theory of computation~Random walks and Markov chains</concept_desc>
	<concept_significance>300</concept_significance>
	</concept>
	</ccs2012>
\end{CCSXML}

\ccsdesc[500]{Computing methodologies~Non-negative matrix factorization}
\ccsdesc[300]{Theory of computation~Random walks and Markov chains}

% We no longer use \terms command
\terms{Theory}

\keywords{Hadamard Product, Row Sum, Column Sum, Closure, Non-Negative}

\maketitle

\section{Introduction}
In this paper, we present a special type of matrix that has a special main property, of which the row sum is equal to the column sum. A well-known \textit{magic matrix} (\cite{Mack},\cite{Swetz}) is a square matrix where the sum of every row is equal to the sum of every column, which is also equal to the sum of diagonal elements. The type of matrix we are introducing is somewhat less restrictive than the magic matrix. Due to some similarity to the magic matrix, we call the matrix as \textit{premagic matrix}. The magic matrix would be a subset of the premagic matrix. 

In this paper, we report the interesting properties of premagic matrix with some applications from ideal flow perspective. Some open problems are also listed.

\section{Literature Review}
In the literature, a model to generate path sequence of the locations of an agent that moves based on random probability distribution is called \textit{random walk}. \textit{Brownian motion} is a subset of random walk where the probability distribution in continuous over time. Random walk on network graph has been studied extensively by \cite{Durrett} and \cite{Grimmett} to mention a few prominent studies. One particular interest within random walk theory is the stationary distribution of nodes. Node stationary distribution is the convergence results of the probability distribution to observe an agent at a node when the random walk is performed at the limit of infinite time step. In simulation of random walk, one could attempt to approximately mimic the infinite time step by setting a large number of time steps. \cite{Blanchard} showed that Random walk on network graph can be studied as Markov Chain. 

\section{Ideal Flow from Random Walk}
Our work on premagic matrix started from the applications of random walk on directed graph for random irreducible networks \cite{Teknomo}. We wanted to come up with a kind of ideal matrices on how the traffic flow in the network should be distributed. We try to imagine that the most distributed flow spatially and temporally may be indicated as the most efficient utilization of the network. Upon observing the properties of the matrices, unifying patterns was found that the matrices of the relative flow are always premagic and can be transformed into whole numbers. 

Suppose we have $N$ random walk agents in an irreducible network over period $T$. The count of the trajectories of the agents on each link is called \textit{link flow} or simply \textit{flow}. \textit{Relative flow} is the link flow divided by a statistical value (i.e. average, sum, min, or max) of all link flows of the network. For simplicity, without losing generality, we will use the minimum flow as the denominator value.
\begin{equation}
\mathbf{F} = \lim_{N.T \to \infty }\frac{\mathbf{R}}{\min {\mathbf{R}}} \text{if } \mathbf{r}_{ij} \neq 0
\label{eq:1}
\end{equation}
The asymptotic of relative flow converged to matrix, which we will called as \textit{ideal flow}.

\begin{definition}
	An \textit{ideal flow matrix} is the asymptotic of relative flow as the results of trajectories of random walk on irreducible network with a given probability distributed over space and time.	
	\label{def:1}
\end{definition}

\section{Premagic Matrix}
The sum of columns can be computed using the matrix multiplication of the row vector one $ \mathbf{j^{T}} $ with a matrix $ \mathbf{A} $ to produce a row vector $ \mathbf{s^{T}} $. To get the vector sum of each row, we multiply from the right. 

\begin{definition}
	Sum of columns (\textit{column sum}) is defined as a row vector by $ \mathbf{s}_{1}^{T}= \mathbf{j}^{T}\mathbf{A} $.
	\label{def:2}
\end{definition}

\begin{definition}
	Sum of rows (\textit{row sum}) is defined as a (column) vector by $ \mathbf{s}_{2}= \mathbf{A}\mathbf{j} $.
	\label{def:3}
\end{definition}

\begin{definition}
A \textit{premagic matrix} is defined either by
\begin{equation}
 \mathbf{s}_{1}= \mathbf{s}_{2} \text{, or}
	\label{eq:2}
\end{equation}
\begin{equation}
  \mathbf{s}_{1}^{T}= \mathbf{s}_{2}^{T} .
	\label{eq:3}
\end{equation}
\label{def:4}
\end{definition}
Premagic matrix $ \mathbf{M} $ is a square matrix which the vector sum of rows (i.e. row sum) is equal to the transpose of the vector sum of columns (i.e. column sum). We can also define premagic matrix by its vector components.

\begin{definition}
Suppose $\mathbf{m}_{j}$  is a vector, and $ \mathbf{m}_{i}^{T}$ is a row vector of a matrix $\mathbf{M}$. A square matrix $\mathbf{M}$ is called \textit{premagic matrix} if $\mathbf{m}_{i}^{T}\mathbf{j}=\mathbf{j}^{T}\mathbf{m}_{j}$.
\label{def:5}
\end{definition}

A matrix whose row sum is equal to the column sum is identical to the matrix whose kernel of the antisymmetric part is a vector one. This also means that any symmetric matrix has a kernel of the antisymmetric part to be a vector one. Let $ T(\mathbf{A}) \mapsto \mathbf{A}-\mathbf{A}^{T} $  be a linear transformation for any square matrix $ \mathbf{A} $. The kernel of the transformation is the set of all vectors $ \mathbf{v} $ such that $ T(\mathbf{A})=0 $ . Thus, the kernel of $ T(\mathbf{A}) $ is a symmetric matrix because $ \mathbf{A}-\mathbf{A}^{T} = \mathbf{0} $  or $ \mathbf{A}=\mathbf{A}^{T} $ . This leads us to an alternative definition of premagic matrix. 

\begin{definition}
	A \textit{premagic matrix} $ \mathbf{M} $ is a matrix whose kernel of the antisymmetric part is a vector one:
	\label{def:6}
	\begin{equation}
	(\mathbf{M}-\mathbf{M}^{T})\mathbf{j}=0
	\label{eq:4}
	\end{equation}
\end{definition}

\section{Properties of Premagic Matrix}
In this section, we prove the theory of premagic matrix by showing its properties. 

Let us first take notes on the trivia cases of the premagic matrix. Any scalar can be considered as a 1 by 1 matrix, and therefore any scalar is a premagic matrix. Any square matrix with equal values for all elements is a premagic matrix. Therefore, the zero square matrix is a premagic matrix. The one square matrix is also a premagic matrix. Furthermore, any diagonal matrix is a premagic matrix. As a special case of diagonal matrix, an identity matrix is a premagic matrix. As mentioned above, any magic matrix is also trivia case of the premagic matrix.

A set has a \textit{closure} under an operation if the result of that operation on the members of the set would always produce a member of the same set. The following propositions and theorems are related to the basic properties of the premagic matrix.

\begin{proposition}
	\label{Proposition:1}
	If $ \{\mathbf{M}_{1},\mathbf{M}_{2},\cdots \mathbf{M}_{n}\} $ are premagic matrices of equal size, $ \mathbf{J} $  is a square matrix of one and $k \in \mathbb{R}$, then the following basic operations are closed.
	\begin{itemize}
		\item Addition Commutative:  
			\begin{equation}
			\mathbf{M}_{1}+\mathbf{M}_{2}=\mathbf{M}_{2}+\mathbf{M}_{1}
			\label{eq:5}
			\end{equation}
	\item Direct Product Commutative: 
		\begin{equation}
		\mathbf{M}_{1} \odot \mathbf{M}_{2}=\mathbf{M}_{2}\odot \mathbf{M}_{1}
		\label{eq:6}
		\end{equation} 
	\item Addition Associative: 
		\begin{equation}
		(\mathbf{M}_{1}+\mathbf{M}_{2})+\mathbf{M}_{3}=\mathbf{M}_{1}+(\mathbf{M}_{2}+\mathbf{M}_{3})
		\label{eq:7}
		\end{equation} 
	\item Direct Product Associative:  
		\begin{equation}
		(\mathbf{M}_{1} \odot \mathbf{M}_{2}) \odot \mathbf{M}_{3}=\mathbf{M}_{1} \odot (\mathbf{M}_{2} \odot \mathbf{M}_{3})
		\label{eq:8}
		\end{equation} 
	\item Distributive:
		\begin{equation}
		(\mathbf{M}_{1} \pm \mathbf{M}_{2}) \odot \mathbf{M}_{3}=(\mathbf{M}_{1} \odot \mathbf{M}_{3}) \pm (\mathbf{M}_{2} \odot \mathbf{M}_{3})
		\label{eq:9}
		\end{equation}	  
	\item Addition:  
		\begin{equation}
		\mathbf{M}_{3}=\mathbf{M}_{1} + \mathbf{M}_{2} 
		\label{eq:10}
		\end{equation}		
	\item Subtraction:  
		\begin{equation}
		\mathbf{M}_{3}=\mathbf{M}_{1} - \mathbf{M}_{2} 
		\label{eq:11}
		\end{equation}		
	\item Scalar multiple:  
		\begin{equation}
		\mathbf{M}_{2} = k \mathbf{M}_{1} 
		\label{eq:12}
		\end{equation}		
	\item Hadamard Product:  
		\begin{equation}
		\mathbf{M}_{2}=\mathbf{M}_{1}  \odot  k\mathbf{J}
		\label{eq:13}
		\end{equation}		
	\item Transpose:
		\begin{equation}
		\mathbf{M}_{2}=\mathbf{M}_{1}^{T} 
		\label{eq:14}
		\end{equation}	  
	\item Shift:  
		\begin{equation}
		\mathbf{M}_{2}=\mathbf{M}_{1} \pm k \mathbf{J}
		\label{eq:15}
		\end{equation}	  
	\item Adding to or subtracting from any diagonal element:  
		\begin{equation}
		\mathbf{M}_{2}=\mathbf{M}_{1} \pm \textup{diag}(\mathbf{J})
		\label{eq:16}
		\end{equation}
	\item Removal of diagonal elements:  
		\begin{equation}
		\mathbf{M}_{2}=\mathbf{M}_{1} - \textup{diag}(\mathbf{M}_{1})
		\label{eq:17}
		\end{equation}
	\item Addition or subtraction with scaled Identity: 
		\begin{equation}
		\mathbf{M}_{2}=\mathbf{M}_{1} \pm k \mathbf{I}
		\label{eq:18}
		\end{equation}	
	\item Linear Combination: 
		\begin{equation}
		\mathbf{M}_{n}=k_{1} \mathbf{M}_{1} \pm k_{2} \mathbf{M}_{2} \pm \cdots \pm k_{n-1} \mathbf{M}_{n-1}
		\label{eq:19}
		\end{equation}	
	\end{itemize}
\end{proposition}

\begin{proof}
	Commutative, associative and distributive properties are simply derived from matrix addition and Hadamard product. From here, it is also clear that square matrix one $ \mathbf{J} $  is the identity of direct product and square matrix zero is the identity of addition. Addition and subtraction of any premagic matrix with any premagic matrix will produce another premagic matrix because the sum of rows and the sum of columns are simply added or subtracted without a change of order. Similarly, the scalar multiple of a real number to any premagic matrix or Hadamard Product of a premagic matrix with a constant matrix will only scale the sum of rows and the sum of columns without changing the order of the rows or the columns. Transposing the operation does not change the sum of rows and the sum of columns, thus the premagic matrix remains premagic after transposing the operation. Adding or subtracting premagic matrix with a constant will increase or decrease the row sum and column sum with the same amount without a change of order.
	Suppose $ m_{ij} $  is a diagonal element with $ i=j $. The sum of columns is denoted by $ s_{i} $. The sum of a non-diagonal element on row $ i $ is $ n_{i}=s_{i}-m_{ii} $. By definition of a premagic matrix, the sum of rows = the sum of columns = $ s_{i} $. Thus, the sum of non-diagonal elements on column $ i $ is also $ n_{i}=s_{i}-m_{ii} $. In other words, what makes a matrix a premagic matrix is the sum of the non-diagonal elements, not the diagonal elements.  Adding or subtracting any diagonal element to a premagic matrix will still produce another premagic matrix because the sum of the non-diagonal elements remain the same. When we remove the diagonal elements of a premagic matrix to make the diagonal elements become zeros, it will still produce a premagic matrix. Thus, adding or subtracting with a scaled identity matrix does not change the order of sums. Since the operations of addition and subtraction and the scalar multiple are closures, linear combination is also closed.	
\end{proof} 

\begin{definition}
	A permutation matrix $ \mathbf{P} $ is (0,1) matrix with a single unit entry in each row and in each column and zeros elsewhere.
	\label{def:7}
\end{definition}

\begin{proposition}
	\label{Proposition:2}
	Multiplication of a premagic matrix by a permutation matrix $ \mathbf{P} $ does not change the premagic properties. Row permutation or column permutation of a premagic matrix produces another premagic matrix.
	\begin{equation}
	\mathbf{M}_{2}= \mathbf{P}^{T} \mathbf{M}_{1} \mathbf{P}
	\label{eq:20}
	\end{equation}
\end{proposition}

\begin{proof}
	Premagic matrix remains the same if only exchange the row or column position because the row sum and column sum is constant when the matrix is permuted.
\end{proof}
		
\begin{theorem}
Any symmetric matrix is a premagic matrix.
\label{Theorem:1}
\end{theorem}

\begin{proof}
	By definition of a symmetric matrix, $ \mathbf{A}=\mathbf{A}^{T} $. Multiplying both sides by vector one $ \mathbf{j}^{T} $, yields $ \mathbf{s}_{1}^{T}= \mathbf{j}^{T}\mathbf{A} = \mathbf{j}^{T}\mathbf{A}^{T} = (\mathbf{A}\mathbf{j})^{T} = \mathbf{s}_{2}^{T}$, which is equal to equation \ref{eq:3}.
\end{proof} 

Note that any symmetric matrix is a premagic matrix but a premagic matrix is not necessarily symmetric. Thus, symmetric matrices can be seen as a subset of a premagic matrix. Many well-known matrices such as Covariance matrix, Correlation Matrix and Laplacian matrix of a graph or a digraph are symmetric matrix that are useful for many applications.

\begin{theorem}
The inverse of a non-singular symmetric matrix is premagic matrix. 
	\label{Theorem:2}
\end{theorem}

\begin{proof}
Suppose $\mathbf{A}$ is a non-singular symmetric matrix. Since the inverse of any symmetric matrix is also symmetric, then the inverse of a symmetric premagic matrix is always premagic if it has an inverse. Then, $\mathbf{A}^{-1}$ is premagic matrix.
\end{proof} 

Any square matrix $\mathbf{C}$  can be decomposed into an antisymmetric part $\mathbf{A}$ and a symmetric part $\mathbf{B}$  such that $\mathbf{C} = \mathbf{A} + \mathbf{B} $ . Where $\mathbf{A}=\frac{1}{2}(\mathbf{C}-\mathbf{C}^{T})$  and $\mathbf{B}=\frac{1}{2}(\mathbf{C}+\mathbf{C}^{T})$ . This leads to the next theorem.

\begin{theorem}
	The kernels of the antisymmetric part of a premagic matrix are always vector one.
	\label{Theorem:3}
\end{theorem}

\begin{proof}
	Suppose $\mathbf{M}$ is a premagic matrix. From equation (\ref{eq:2}) we have
	\[ \mathbf{s}_{1}= \mathbf{s}_{2} \]
	\[ (\mathbf{j}^{T}\mathbf{M})^{T} = \mathbf{M}\mathbf{j} \]
	\[ \mathbf{M}^{T}\mathbf{j} = \mathbf{M}\mathbf{j} \]
	\[ (\mathbf{M}^{T}-\mathbf{M})\mathbf{j} = \mathbf{0} \]
	Thus, the kernel of the antisymmetric part of the premagic matrix $\mathbf{M}$  is always a vector one $\mathbf{j}$.
\end{proof} 

Any symmetric matrix can be decomposed into a direct product of a matrix and its transpose. This direct product decomposition is not unique.
 
\begin{theorem}
	\label{Theorem:4}
	Direct product (Hadamard product ) of a square matrix  with its transpose always produces a symmetric matrix with a square of diagonal elements
	\begin{equation}
	\mathbf{B}= \mathbf{A} \odot \mathbf{A}^{T} \Leftrightarrow \mathbf{B}= \mathbf{B}^{T}
	\label{eq:21}
	\end{equation}	
\end{theorem}

\begin{proof}
	Let $ \mathbf{A}=\begin{bmatrix}
		a_{11} & a_{12} & \cdots  & a_{1n}\\ 
		a_{21} & a_{22} & \cdots  & a_{2n}\\ 
		\vdots & \vdots & \ddots  & \vdots\\ 
		a_{n1} & a_{n2} & \cdots  & a_{nn}\\ 
	\end{bmatrix} $
	then $\mathbf{B}= \mathbf{A} \odot \mathbf{A}^{T} = \begin{bmatrix}
		a_{11}^{2}   & a_{12}a_{21} & \cdots  & a_{1n}a_{n1}\\ 
		a_{12}a_{21} & a_{22}^{2}   & \cdots  & a_{2n}a_{n2}\\ 
		\vdots       & \vdots       & \ddots  & \vdots\\ 
		a_{1n}a_{n1} & a_{2n}a_{n2} & \cdots  & a_{nn}^{2}\\ 
	\end{bmatrix}$, which is clearly symmetric matrix because of the commutative property ($ a_{ij}a_{ji} = a_{ji}a_{ij} $). Now, if we have the symmetric matrix $ \mathbf{B}  $ , we can find the standard matrix $ \mathbf{A}  $ by taking the square root of each element of $ \mathbf{B}  $ , that is $ \mathbf{A} = \sqrt{\mathbf{B}}  $
\end{proof} 

\begin{corollary}
	\label{corollary:1}
	The Hadamard Product of a premagic matrix with its transpose is closed.
	\begin{equation}
	\mathbf{M}_{2}= \mathbf{M}_{1} \odot \mathbf{M}_{1}^{T} 
	\label{eq:22}
	\end{equation}
\end{corollary}

\begin{proof}
	Based on Theorem \ref{Theorem:4}, the Hadamard Product of a square matrix with its transpose always produces a symmetric matrix with a square of diagonal elements. Symmetric matrices are always premagic, based on Theorem \ref{Theorem:1}.
\end{proof} 

\begin{theorem}
	\label{Theorem:5}
	A product of any rectangular matrix $ \mathbf{A} $ and its transpose is premagic.
	\begin{equation}
	\mathbf{M}= \mathbf{A}\mathbf{A}^{T} 
	\label{eq:23}
	\end{equation}
\end{theorem}

\begin{proof}
	A product of any rectangular matrix and its transpose always produces a symmetric matrix. Symmetric matrices are always premagic based on Theorem \ref{Theorem:1}.
\end{proof} 

\begin{theorem}
	\label{Theorem:6}
	If the non-diagonal elements of a matrix has equal value, then that matrix is premagic.
\end{theorem}

\begin{proof}
	We start with a trivia case that a square matrix that has equal values for all elements is a premagic matrix. Based on Equation (\ref{eq:16}), adding any diagonal element to a premagic matrix will produce another premagic matrix; thus, if the non-diagonal elements of a matrix has equal value, then that matrix must be premagic. 
\end{proof} 

\begin{theorem}
	\label{Theorem:7}
	The 1-norm of a non-negative premagic matrix is always equal to its infinity-norm and it is equal to the largest sum of rows or the largest sum of columns.
	\begin{equation}
	\left \| \mathbf{A} \right \|_{1} = \left \| \mathbf{A} \right \|_{\infty } = \mathbf{s}_{1}
	\label{eq:24}
	\end{equation}	
\end{theorem}

\begin{proof}
	1-norm is the maximum absolute column sum. The norm is formulated as $ \left \| \mathbf{A} \right \|_{1}=  \underset{j}{\max}\sum_{i=1}^{n} \left | a_{ij} \right | $. The infinity-norm is the maximum absolute row sum. The norm is formulated as $ \left \| \mathbf{A} \right \|_{\infty } = \underset{i}{\max}\sum_{j=1}^{n} \left | a_{ij} \right | $. Since the premagic matrix has an equal sum of rows and sum of columns, the 1-norm is equal to the infinity norm and is equal to the largest sum of rows or the largest sum of columns. This is only true for premagic with non-negative elements. When some of the elements of the premagic matrix are negative, the sum of the absolute value will not produce the same sum as the sum of rows or the sum of columns, thus the proposition is only true for positive or zero element matrices. 
		
\end{proof} 

\begin{theorem}
	\label{Theorem:8}
	Suppose we have $ 2 \times 2 $ premagic matrix \[ 
	\mathbf{M}=\begin{matrix}
		\begin{bmatrix}
			a & b \\ 
			c & d
		\end{bmatrix} & \begin{matrix}
		e_{1}  \\ 
		e_{2} 
	\end{matrix}\\ 
	\begin{matrix}
		e_{1}  & e_{2} 
	\end{matrix} & t
\end{matrix}
 \]
 Then, the following are true:
 \begin{itemize}
 	\item If a premagic matrix   order 2 has an inverse, then the inverse is always premagic.
 	\item $ t+\left | \mathbf{M} \right |=-b^{2}+2b+\left ( a+d+ad \right ) $ which is a quadratic equation
 	\item $ e_{1}e_{2}+\left | \mathbf{M} \right |=c+\left ( a+d \right ) +2ad $
 \end{itemize}
\end{theorem}

\begin{proof}
	The first row of premagic matrix $ \mathbf{M} $ is equal to the first column, $ a+b=e_{1}=a+c  \rightarrow b=c $. The second row of premagic matrix $ \mathbf{M} $ is equal to the second column, $ b+d=e_{2}=c+d  \rightarrow b=c $. Adding the two equations above, we have $ t=e_{1}+e_{2}=a+2b+d $. Since $ b=c $, the determinant is $ \left | \mathbf{M} \right |=ad-b^{2} $. Adding $ t+\left | \mathbf{M} \right |= \left ( a+2b+d \right ) + \left ( ad-b^{2} \right )= -b^{2}+2b+\left ( a+d+ad \right ) $. Multiplying the sum and adding it to the determinant, we have
	\[ e_{1}e_{2}+\left | \mathbf{M} \right | = \left ( a+c \right ) \left ( c+d \right )+ad-c^{2} \] 
	\[ =ac+c^{2}+ad+cd+ad-c^{2} \]
	\[ =c+\left ( a+d \right ) +2ad \].
	Premagic matrix $ \mathbf{M} $  has no inverse only if $ \left | \mathbf{M} \right |=ad-b^{2} = 0 $. Thus, $ b=\sqrt{ad} $. If the Premagic matrix $ \mathbf{M} $  has an inverse, then $ b^{2}\neq ad $. Thus, the inverse is $ \mathbf{M}_{-1}=\frac{1}{ad-b^{2}}\begin{bmatrix}
	d & -b\\ 
	-b & a
	\end{bmatrix} $, which is always premagic because the off-diagonal elements are equal to $ b $.
\end{proof}

\section{Relationship of Ideal Flow and Premagic Matrix}
In this section, we will describe our main result of the relationship between Ideal Flow matrix and Premagic matrix. 

\begin{theorem}
	\label{Theorem:9}
	Non-negative Premagic matrices represent node conservation flow in an irreducible directed network graph.
\end{theorem}

\begin{proof}
	Suppose we have a strongly connected network directed graph (irreducible) where the value of each link indicates flow. The sum of rows represent the sum of the outflow of a node $ \sum f_{out} $ and the sum of columns represent the sum of the inflow to a node $ \sum f_{in} $. Node conservation of flow means the sum of the inflow is equal to the sum of the outflow $ \sum f_{out} = \sum f_{in} = \upsilon $. Thus, if the flow on a node is conserved, then the matrix must be premagic and vice versa.  
\end{proof} 

\begin{corollary}
	\label{corollary:2}
	Ideal flow matrix is always premagic.
\end{corollary}

\begin{proof}
	If matrix $ \mathbf{M} $ is an ideal flow, then $ n_{i}= \sum_{j}  m_{ij} $ is the sum of inflow to node $ i $ and $ n_{j} = \sum_{i}  m_{ij} $ is the sum of outflow from node $ j $. If the flows on all nodes are conserved and the network is strongly connected, then the matrix must be premagic.
\end{proof}

\begin{theorem}
	\label{Theorem:10}
	Ideal flow matrix is converged and can be transformed into whole numbers.
\end{theorem}

\begin{proof}
	To achieve node stationary distribution to form transition probability matrix, which has Markov property $ \mathbf{P} = \left [ p_{ij} \right ] = Pr\left (v_{t+1}=j |v_{t}=i \right )> 0, i\neq j $. For undirected graph, Transition Probability Matrix is a symmetric matrix and equal to $ \mathbf{P} = \mathbf{D}^{-1}\mathbf{A} $, where matrix $ \mathbf{D} $ is a diagonal matrix of each node degree and $ \mathbf{A} $ is adjacency matrix of the graph. The vector of stationary distribution of nodes is called Perron's vector denoted by $  \mathbf{\pi} $. Perron-Frobenius theorem (\cite{Seneta}) states that the left Eigen vector of the transition matrix is associated to the maximal eigenvalue $ \lambda =1 $ such that $ \mathbf{\pi P}=1\mathbf{\pi}  $. \cite{Blanchard} showed that the equation above converged for irreducible network. Since ideal flow is a relative flow, they are rational number and since multiplying with a scalar (as in equation (\ref{eq:12})) as does not change the properties of premagic matrix, we can multiply with the Least Common Multiple (LCM) of the denominators of the elements of the matrix to produce matrix of whole numbers with the same meaning. 
\end{proof}

The following two propositions can be used to compute premagic matrix directly without random walk simulation of the ideal flow.

\begin{theorem}
	\label{Theorem:12}
	Non-negative premagic matrices can be converted to Markov Transition Matrix and vice versa.
\end{theorem}

\begin{proof}
	Suppose we have premagic matrix $ \mathbf{M}=\left [ m_{ij} \right ]  $. The elements of the sum of row is $ n_{i}= \sum_{j}  m_{ij}. $. The elements of the sum of column is $ n_{j}= \sum_{i}  m_{ij}. $. The premagic property stated that $ n_{i}= n_{j}  \forall i=j $.  Markov (row stochastic) transition matrix is obtained using
	\begin{equation}
	\mathbf{S}=\left [ s_{ij} \right ] = \frac{m_{ij}}{\sum_{j} m_{ij} }= \frac{m_{ij}}{n_{i}}
	\label{eq:25}
	\end{equation}
	To obtain a premagic matrix from Markov transition matrix, simply use
	\begin{equation}
	\mathbf{M}=\left [ m_{ij} \right ] = n_{i}s_{ij}
	\label{eq:26}
	\end{equation}
	Note that Markov transition matrix itself is not necessarily premagic. 
\end{proof}

\begin{theorem}
	\label{Theorem:13}
	Non-negative premagic matrices can be converted to double stochastic Markov transition matrix and vice versa.
\end{theorem}

\begin{proof}
	Double Stochastic Markov Transition matrix can be obtained by dividing each entry of a premagic matrix with the total of all elements in the matrix
	\begin{equation}
	\mathbf{S}=\left [ s_{ij} \right ] = \frac{m_{ij}}{\sum_{i}\sum_{j} m_{ij} }
	\label{eq:27}
	\end{equation}
	The double stochastic Markov transition matrix obtained through equation (\ref{eq:27}) itself is a premagic matrix due to the Shift property in equation (\ref{eq:15}). Knowing the total flow in the system $ \kappa=\sum_{i}\sum_{j} m_{ij} $, we can obtain a premagic matrix from double stochastic Markov transition matrix as
	\begin{equation}
	\mathbf{M}=\left [ m_{ij} \right ]=\kappa s_{ij}
	\label{eq:28}
	\end{equation}
\end{proof}

\section{Open Problems}
Here are our conjectures that can be found using a computer simulation but they need mathematical proofs.
\begin{conjecture}
	For a non-negative premagic matrix, both the 2-Norm and the Frobenius norm of the eigenvector of the dominant eigenvalue are equal to one.
\end{conjecture}
\begin{conjecture}
	Premagic matrix is diagonalizable. The eigenvectors of a premagic matrix has full rank.
\end{conjecture}
The Perron Frobenius theorem (see the proof in \cite{Seneta}) gives us a clue that a simple (unique, non-repeating) positive real eigenvalue exists for a non-negative square matrix. The Perron Frobenius theorem also says that the real eigenvalue is dominant, and that the corresponding eigenvector is real and positive.  
\begin{conjecture}
	Suppose $ \lambda_{1} $  is the dominant eigenvalue of non-negative premagic matrix $ \mathbf{M}_{1} $  and $ \mathbf{v} $  is the corresponding eigenvector. If we obtain $ \mathbf{M}_{2} = k \mathbf{M}_{1} k\in \mathbb{R}^{+} $, then the dominant eigenvalue of $ \mathbf{M}_{2} $  is $ \lambda_{2}=k\lambda_{1} $ and the corresponding eigenvector is $ \mathbf{v} $.
\end{conjecture}

\section{Conclusion}
We established the theory of premagic matrix through its properties. 
Ideal flow matrix, which originally comes from random walk on strongly connected network, is always premagic matrix because flow conservation on nodes. Using the premagic property, we showed that the ideal flow matrix can be obtained from Markov Matrix. Using Markov Chain, we generalize the results of random walk to go beyond uniform distribution.
The ideal flow can be applied to any strongly connected networks such as transportation or communication networks.

\begin{acks}
This research is supported by the Philippine Higher Education Research Network (PHERNET).
The author also would like to thank Wiwat Wanicharpichat and Peter T. Brauer for useful discussion in Research Gate.
\end{acks}

\vspace*{\fill}

\end{document}